\documentclass[a4paper,10pt]{article}

\usepackage{geometry}
\usepackage{amsmath}
\usepackage{amssymb}
\usepackage{amsthm}
\usepackage[latin1]{inputenc}
\usepackage{eurosym}
\usepackage[dvips]{graphics}
\usepackage{graphicx}
\usepackage{epsfig}

\usepackage[hypertex]{hyperref}

\usepackage{ifthen}

\newcommand{\argmin}{\operatorname{argmin}}
\newcommand{\argmax}{\operatorname{argmax}}
\newcommand{\diag}{\operatorname{diag}}

\newcommand{\Rr}{{\mathbb{R}}}

\newcommand{\Gg}{{\mathcal{G}}}
\newcommand{\Kk}{{\mathcal{K}}}
\newcommand{\Pp}{{\mathcal{P}}}

\newcommand{\Sss}{{\mathbb{S}}}

\newtheorem{teo}{Theorem}

\newtheorem{lem}{Lemma}
\newtheorem{pro}{Proposition}
\newtheorem{definition}{Definition}
\newtheorem{hypothesis}{Assumption}

 \newcommand{\re}{\mathbb{R}}

\newcommand{\ep}{\epsilon}

\begin{document}

\title{Discrete mean field games}
\author{Diogo
  A. Gomes
\footnote{Departamento de Matem\'atica and CAMGSD, IST, Lisboa, Portugal. e-mail: dgomes@math.ist.utl.pt}, Joana Mohr \footnote{Instituto de Matem\'atica, UFRGS, 91509-900 Porto Alegre, Brasil. e-mail: rafars@mat.ufrgs.br} and Rafael Rig\~ao Souza \footnote{Instituto de Matem\'atica, UFRGS, 91509-900 Porto Alegre, Brasil. e-mail: joana.mohr@ufrgs.br}}

\date{\today} 

\maketitle

\thanks{D.G. was
partially supported by the CAMGSD/IST through FCT Program POCTI/FEDER
and by grant DENO/FCT-PT (PTDC/EEA-ACR/67020/2006).}

\thanks{R.R.S was
partially supported by CAPES, PROCAD, Projeto Universal CNPq 471473/2007-3.}

\thanks{J.M: was beneficiary of a CNPq PhD scholarship and now has a Post-Doc CNPq scholarship.}

\begin{abstract}
In this paper we study a mean field model for discrete time,
finite number of states,
dynamic games. These models arise in situations
that involve a very large number of agents
moving from state to state according to certain optimality criteria.
The mean field approach
for optimal control and differential games (continuous state and time)
was introduced by Lasry and Lions \cite{ll1, ll2, ll3}.
The discrete time, finite state space setting is motivated both by
its independent interest as well as by numerical
analysis questions which appear in the discretization of the problems introduced by
Lasry and Lions.

Our setting is the following: we assume that there is a very large number of identical agents
which can be in a finite number of states. Because the number of agents is
very large, we assume the mean field hypothesis, that is, that the only relevant information for the global evolution
is the fraction $\pi^n_i$ of players in each state $i$ at time $n$.
The agents look for minimizing a
running cost, which depends on $\pi$, plus a terminal cost $V^N$.
In contrast with optimal control, where usually only the terminal cost $V^N$ is necessary to solve the problem,
in mean-field games both the initial distribution of agents $\pi^0$ and the terminal cost $V^N$ are necessary to determine
the solutions, that is,  the distribution of players $\pi^n$ and value function
$V^n$, for $0\leq n\leq N$. Because both initial and terminal data needs to be specified, we call this problem
the initial-terminal value problem. Existence of solutions is non-trivial. Nevertheless, following the ideas of Lasry and Lions,
we were able to
establish existence and uniqueness, both for the stationary and for the initial-terminal value problems. We discuss
in some detail a particular model, the entropy penalized problem.
In the last part of the paper we prove the main result of the paper, namely the exponential convergence to a stationary solution
of $(\pi^0, V^0)$, as $N\to \infty$, for
the initial-terminal value problem with (fixed) data $\pi^{-N}$ and $V^N$.
\end{abstract}


\section{Introduction}

In this paper we study a mean field model for discrete time,
finite number of states,
dynamic games. These models arise in situations
that involve a very large number of agents
moving from state to state according to certain optimality criteria.
The mean field approach
for optimal control and differential games (continuous state and time)
was introduced by Lasry and Lions \cite{ll1, ll2, ll3}.
In the continuous state and time setting, mean field problems
gives rise to Hamilton-Jacobi equations coupled with transport equations.
The discrete time, finite state space setting is motivated both by
its independent interest as well as by numerical
analysis questions which appear in the discretization of the problems introduced by
Lasry and Lions. The discretization of these models
has been studied by I. Capuzzo-Dolcetta and Y. Achdou.

Our setting is the following: we assume that there is a very large number of identical agents
which can be in a finite number of states.
Each agent behaves individually and rationally,
moving from state to state according to certain optimality criteria.
Furthermore,
its decisions are based solely on the following information,
which is know by every agent,
the current state, and the fraction of agents in each state.
As in non-cooperative games, there may be interactions between the players
in different states, as we will explain in more detail bellow.
Because the number of agents is
very large, we assume the mean field hypothesis, that is, that only the fraction $\pi^n_i$ of players in each state $i$ at time $n$
is the relevant information for the global evolution. The mathematical justification of mean field models
has been investigated extensively by Lions and Lasry in yet to be published papers and we do not address these issues in this work.


Let $d>1$ and $N\geq 1$ be natural numbers, representing, respectively,
the number of possible states in which the agents can be at any given time, and the total duration of the process.
Let $\pi^0$ and $V^N$ be given $d$-dimensional vectors.
We suppose that $\pi^0$ is a probability vector, the initial probability distribution of agents among states,
and that $V^N$, the terminal cost, is an arbitrary vector.
A solution to the mean field game is a sequence of pairs of $d$-dimensional vectors
$$\{(\pi^n,V^n) \;  ; \;0\leq n \leq N \} \, ,$$
where $\pi^n$
is the probability distribution of agents among states at time $n$
and
$V^n_j$ is the expected minimum total cost for an agent at state $j$, at time $n$. These pairs must satisfy certain
optimality conditions that we describe in what follows:
at every time step, the agents in state $i$ choose a transition probability, $P_{ij}$,
from state $i$ to state $j$.
Given the transition probabilities $P_{ij}^n$ at time $0\leq n<N$, the distribution of agents at time $n+1$ is simply
\[
\pi^{n+1}_j=\sum_i\pi^n_i  P_{ij}^n \,.
\]
Associated to this choice there is a transition cost $c_{ij}(\pi, P)$. In the special case in which $c_{ij}$ only depends on $\pi$
and on the $i$th line of $P$ we use the simplified notation $c_{ij}(\pi, P_{i\cdot})$. This last case
arises when the choices of players in states $j\neq i$ do not influence the transition cost
to an agent in state $i$.
Let $e_i(\pi,P, V)$ be the average cost that agents which are in state $i$ incur when matrix $P$ is chosen, given the current distribution $\pi$ and the cost vector $V$ at the subsequent instant.
We assume that
\[
e_i(\pi, P, V)=\sum_j c_{ij}(\pi, P)P_{ij}+V_j P_{ij}.
\]
Define the probability simplex
$\mathbb{S} = \{(q_1,...,q_d)\; ;\; q_{j} \geq 0 \, \forall j, \sum_{j=1}^d q_{j} = 1 \}$.
The set of $d \times d$ stochastic matrices is identified with $\mathbb{S}^d$.
Given a  stochastic matrix $P \in \mathbb{S}^d$ and a probability vector $q \in \mathbb{S}$, we define
$\Pp(P,q,i)$
to be the $d \times d$ stochastic matrix obtained from $P$ by replacing the $i$-th row by $q$, and leaving all others unchanged.
\begin{definition} Fix a probability vector $\pi\in \mathbb S$ and a cost vector $V\in \Rr^d$.  A stochastic matrix $P\in \mathbb{S}^d$
is a {\em Nash minimizer} of $e(\pi, \cdot, V) $ if for each $i\in \{1,...,d\}$ and any $q\in \Sss$
$$e_i(\pi,P,V)\leq e_i(\pi,\Pp(P,q,i),V).$$
\end{definition}

\begin{definition} Suppose that for each $\pi\in \mathbb{S}$ and $V\in \Rr^d$ there exists a Nash minimizer $P\in \mathbb{S}^d$ of $e(\pi, \cdot, V)$.
Let $N\geq 1$, $\pi^0\in \mathbb{S}$ (the initial distribution of states), and $V^N \in \re^d$ (the terminal cost).

A sequence of pairs of d-dimensional vectors  $$\{(\pi^n,V^n) \;  ; \;0\leq n \leq N \} \, $$ is a {\em solution of the mean field game} if
for every $0\leq n \leq N-1$
\begin{equation}
\label{mfg}
\left\{
\begin{array}{rcl}
\displaystyle V^n_i&=&\sum_j c_{ij}(\pi^n,P^n)P_{ij}^n+V^{n+1}_jP_{ij}^n\\\\
\displaystyle \pi^{n+1}_j&=&\sum_i \pi^n_i P_{ij}^n,
\end{array}
\right.
\end{equation}
for some Nash minimizer $P^n\in \mathbb{S}^d$  of $e(\pi^n, \cdot, V^{n+1}) $.
\end{definition}

Until the end of this section we will assume that for all $(\pi, V)\in \Sss\times \Rr^d$ there exists a unique Nash minimizer $\bar P$ of $e(\pi,P,V)$.
Conditions which guarantee the uniqueness of a Nash minimizer will be studied in \S \ref{dcc}. 
Under the uniqueness of a Nash minimizer for $e$, we can define the (backwards) evolution operator for the value function
\begin{equation}
\label{ggdef}
\Gg_\pi(V)=e(\pi,\bar P,V),
\end{equation}
as well as the (forward) evolution operator for $\pi$
\begin{equation}
\label{kkdef}
\Kk_V(\pi)=\pi \bar P.
\end{equation}
Since the operator $\Gg_\pi$ commutes with addition with constants, it can be regarded as a map from $\Rr^d/\Rr$ to  $\Rr^d/\Rr$. Here
$\Rr^d/\Rr$ is the set of equivalence classes of vectors in $\Rr^d$ whose components differ by the same constant.
In $\Rr^d/\Rr$ we define  the norm
\begin{equation}
\label{sharpnorm}
\|\psi \|_{\#}:=\inf_{\lambda\in\re}\|\psi +\lambda \|,
\end{equation}
In this paper
we will regard $\Gg_\pi$, depending on what is convenient,  as both a map in $\Rr^d$ as well as a map in $\Rr^d/\Rr$.

We have the compact equivalent form for \eqref{mfg}
\begin{equation}
\label{mfgc}
\left\{
\begin{array}{rcl}
\displaystyle V^n&=&\Gg_{\pi^n}(V^{n+1})\\\\
\displaystyle \pi^{n+1}&=&\Kk_{V^{n+1}}(\pi^n).
\end{array}
\right.
\end{equation}
In this paper we will consider solutions to \eqref{mfgc} which satisfy initial-terminal value conditions, $\pi^0$ (or $\pi^{-N}$) and $V^N$,
as well as stationary solutions, that we discuss in what follows.

\begin{definition}
A pair of vectors $(\bar \pi, \bar V)$ is a stationary solution to the mean field game if there exists a constant $\bar \lambda$, called critical value, such that
\begin{equation}
\label{stat}
\left\{
\begin{array}{rcl}
\displaystyle \bar V&=&\Gg_{\bar \pi}(\bar V)+\bar \lambda\\\\
\displaystyle \bar \pi&=&\Kk_{\bar V}(\bar \pi).
\end{array}
\right.
\end{equation}
\end{definition}

We should remark that the first equation in \eqref{stat} can be written in $\Rr^d/\Rr$ as $\Gg_{\pi}(\bar V)=\bar V$. Therefore,
solutions to \eqref{stat} can be regarded as fixed points of $(\Gg_\pi, \Kk_V)$ in $\Rr^d/\Rr\times \Sss$.




The structure of the paper is as follows:
we first start, in \S \ref{mh} by listing our main assumptions, as well as explaining where they are needed in the paper.
In \S \ref{ttm} we address the issue of existence, theorem
\ref{existgeneral},
 and uniqueness, theorem \ref{teoA}, of
Nash-minimizing transition matrices.
Some general properties of the operator $\Gg$ are studied in \S\ref{gprop}.
In \S \ref{statsol} we establish several results concerning the existence,
theorem \ref{teoB}, and uniqueness, propositions \ref{ups}, \ref{ucv} and theorem \ref{teoC}, of stationary solutions to mean field games.
The initial-terminal value problem is studied in \S \ref{esmfg}. We show also existence, theorem \ref{teoD}, and uniqueness, theorem \ref{teoE},
of solutions for this problem.
Both in the stationary and initial-terminal value
problem the uniqueness proofs use a version of the mononicity argument of Lasry and Lions in \cite{ll1, ll2, ll3}.
In the last section we address the convergence to equilibrium and establish one of the main
results of the paper, theorem \ref{teoF}, which states that, as we take the initial and terminal conditions far appart ($n=\pm N$), the solutions at $n=0$
converge exponentially to a stationary solution.
Throughout the paper we discuss with detail the entropy penalized problem
For this model we give an independent proof of existence of stationary solutions, proposition \ref{sep} and
study the large entropy limit, proposition \ref{mainproA}.
Another important example is the optimal stationary solutions, discussed in
\S \ref{oss2} which give
a variational interpretation of a solution of certain mean field games in terms of non-linear programming problems, proposition
\ref{mainproB}.

\section{Main Assumptions}
\label{mh}

In this section, for the convenience of the reader, we list the main assumptions that will be needed in the text.
We list here only the assumptions that will be used in the main results or the ones which are
repeatedly used in the text.
A few other assumptions will be introduced later in the text and will only be used "locally" in the section they are stated.

The first two assumptions will be used in theorem \ref{existgeneral}, \S\ref{tgc},  to establish existence of a Nash minimizer of $e$.
\begin{hypothesis}
\label{hp1}
For each $\pi\in \Sss$, $V\in \Rr^d$, $P\in \Sss^d$, and each index $1\leq i\leq d$,
the mapping 
$q\mapsto e_i( \pi, \Pp(P,q,i) , V)$,
defined for $q\in\mathbb{S}$, and taking values on $\mathbb{R}$, is
convex.
\end{hypothesis}
\begin{hypothesis}
\label{hp2}
The map $P\mapsto e_i(\pi,P,V)$ is
continuous for all $i$.
\end{hypothesis}

Concerning the uniqueness of Nash minimizers, addressed in \S\ref{dcc}, we need the following definition (see \cite{Haurie}
for the motivation of this assumption):
\begin{definition}
A function $g:\Rr^{d\times d} \to \Rr^{d\times d}$ is diagonally convex
if for all $P^1, P^2\in \Rr^d$, $P^1\neq P^2$, we have
\[
\sum_{ij} (P^1_{ij}-P^2_{ij}) (g_{ij}(P^1) -g_{ij}(P^2)) >0.
\]
\end{definition}
With this definition we can state the next assumption:
\begin{hypothesis}
\label{hp3}
Let
\[
g_{ij}(P)=\frac{\partial e_i(\pi, P, V)}{\partial P_{ij}} \,.
\]
Then $g$ is diagonally convex.
\end{hypothesis}

Since diagonal convexity may not be the unique way to ensure uniqueness of Nash minimizers,
it is convenient to add  uniqueness as an assumption, which obviously holds under assumptions \ref{hp1}-\ref{hp3}, but which can also hold
under other alternative hypothesis.
\begin{hypothesis}
\label{hp4}
For  each $(\pi, V)$ there exists a unique Nash minimizer $P(\pi, V)$
of $e(\pi, \cdot, V)$.
\end{hypothesis}
The uniqueness of $P(\pi, V)$ makes the operators $\Kk_V$ and $\Gg_\pi$ well defined. Therefore, from \S \ref{statsol} on we will always suppose
that assumption \ref{hp4} holds, even without explicit mention.

To establish continuity of $P(\pi, V)$, \S \ref{uac}, proposition \ref{cont}, we need:
\begin{hypothesis}
\label{hp5}
For each index $1\leq i\leq d$,  $e_i: \Sss\times \Sss^d\times \Rr^d\to \Rr$ is a continuous function.
\end{hypothesis}

Denote by $\rho_{i,i'}(P)$ the matrix we
obtain from $P$ by replacing its $i'$-th row by its $i$-th row, and leaving all other rows (including the $i$-th) unchanged.
\begin{hypothesis}
\label{hp6}
There exists $C>0$ such that
for all $i$ and $i'$,
and any $\pi\in \Sss$, $P\in \Sss^d$
\begin{equation}
\label{semiconcavity}
\sum_j \left|\,c_{ij}(\pi, P)-c_{i'j}(\pi, \rho_{i,i'}(P))\right|P_{ij}\leq C.
\end{equation}
\end{hypothesis}

Note that the previous assumption holds if $c_{ij}$ is bounded, for instance.

Some of the results of the paper only hold for transition costs which have a special dependence on $P$.
The next assumption will be required frequently:
\begin{hypothesis}
\label{hp7}
The cost $c_{ij}(\pi, P_{i\cdot})$ depends on $\pi$ and, for each $i$, only on the $i$-th line of P.
\end{hypothesis}

To establish uniqueness of solutions (\S \ref{uniq}, theorem \ref{ups}, \S \ref{usitvp}, theorem \ref{teoE}),
as well as to obtain exponential convergence to stationary solutions (\S\ref{cvte}, theorem \ref{teoF})
 it is convenient
to have the following assumption on the operator $\Gg$:
\begin{hypothesis}
\label{hp8}
There exists a constant $\gamma>0$ such that
\[
\tilde \pi \cdot (\Gg_{\tilde \pi}(V)-\Gg_{\pi}(V))
+\pi \cdot (\Gg_{\pi}(\tilde V)-\Gg_{\tilde \pi}(\tilde V))
\geq \gamma \|\pi-\tilde \pi\|^2,
\]
for any $V, \tilde V\in \Rr^d$ and all $\pi, \tilde \pi\in \Sss$.
\end{hypothesis}

An example where this last hypothesis is satisfied is the following:
\begin{equation}
\label{onlypi}
c_{ij}(\pi, P_{i\cdot})=W_i(\pi)+\tilde c_{ij}(P_{i\cdot}),
\end{equation}
where $W$  is a monotone function, that is,
\begin{equation}
\label{monotonew}
(\pi-\tilde \pi) \cdot (W(\pi)-W(\tilde \pi)) \geq \gamma \|\pi-\tilde \pi\|^2,
\end{equation}
where $\gamma$ is a positive constant. For instance, the gradient of a convex function is a monotone function.
In this case we have
\[ \Gg_{\pi}(V)(i)=  
W_i(\pi) + \min_{P_{ij}} \sum_j (\tilde c_{ij}(P_{i\cdot})+V_j)P_{ij}
\,.\]
The special structure in \eqref{onlypi} arises naturally in certain problems, see \S \ref{oss2}.

To establish the uniqueness of the critical value we need the following assumption:
\begin{hypothesis}
\label{hp9}
For any $\pi\in \Sss$, the operator
$V\mapsto \Gg_\pi(V)$ satisfies the following property:
for all $V^1$, $V^2$ and any $i\in \argmax (V^1-V^2)$ we have
\[
\Gg_\pi(V^1)_i-V^1_i\leq \Gg_\pi(V^2)_i-V^2_i,
\]
with the opposite inequality if $i\in \argmin(V^1-V^2)$.
\end{hypothesis}
As it will be proved in \S\ref{gmonot}, proposition \ref{monotoneprop}, assumption \ref{hp7} implies assumption \ref{hp9}. However, we leave it
explicit to make easier the understanding of what follows.


The following strict concavity of $\Gg$ is important to establish uniqueness of stationary solutions and the exponential convergence to equilibrium.
\begin{hypothesis}
\label{hp10}
For all $\pi\in \Sss$ and all $V^1, V^2\in \Rr^d$ we have
\[
\pi \cdot (\Gg_{\pi}(V^2)-\Gg_{\pi}(V^1))+\Kk_{V^1}(\pi) (V^1-V^2)\leq -\gamma_{\pi} \|V^1-V^2\|^2_\#.
\]
\end{hypothesis}
This last assumption is a slightly stronger version of inequality \eqref{kav3}, which is a consequence of
assumptions \ref{hp4} and \ref{hp7}.

A final hypothesis allow us to establish certain bounds (lemma \ref{bound} in \S \ref{apb2}) which are useful in proving the  exponential convergence to equilibrium:
\begin{hypothesis}
\label{hp11}
There exists $K>0$ such that
for  all $\pi,\tilde \pi\in \Sss$, and for any matrix  $P\in \Sss^d$
\begin{equation}\label{lema_conveq_hip}\left|\,c_{ij}(\pi, P)-c_{ij}(\tilde\pi,P)\right|\leq K.\end{equation}
\end{hypothesis}
Note that the previous assumption holds if $c_{ij}$ is bounded, for instance.

\section{The Transition Matrix}
\label{ttm}

In this section we discuss the problem of existence, uniqueness and continuity
of the Nash equilibrium transition matrix $P_{}$.
As we will see, this problem is non-trivial and requires, in the general case,
the use of Kakutani's fixed point theorem, see theorem \ref{existgeneral} in \S \ref{tgc}.
Once existence is established,
uniqueness can be proven for a general class of cost functionals, theorem \ref{teoA} in \S \ref{dcc}.
We finish this section with the discussion of a special
 case: the entropy penalized model, \S \ref{epmd}.

\subsection{Existence in the general case}\label{tgc}

\begin{teo}\label{existgeneral}
Suppose that assumptions \ref{hp1} and \ref{hp2} hold.
Then, for any pair of vectors $\pi$ and $V$ there exists a Nash minimizer $P$ of $e(\pi, \cdot, V)$.
\end{teo}
\begin{proof}
Given a stochastic matrix $P$, define $F_i(P)$ to be the set of vectors in $\mathbb{S}$	given by
$$ F_i(P)=\argmin_{q\in\mathbb{S}} \;e_i(\pi,\Pp(P,q,i),V).$$
The set $F_i(P)$ is non-empty and convex.
Define $F(P) = F_1(P) \times F_2(P) \times ... \times F_d(P)$, where we
identify the cartesian product with the the set of all stochastic matrices where the $i$-th row belongs to $F_i(P)$.
Clearly, $F(P)$ is convex for all $P$.
Furthermore, as we argue next,
the graph $\{(P;F(P)) ; P \in \mathbb{S}^d\}$ is closed.
Indeed,
suppose that $ P^n\to P^0$ and take $Q^n\in F(P^n)$, if $Q^n\to Q^0$
we want to show that $Q^0\in F(P^0)$. Fix $i$ and call $q^n_i$
the i-th coordinate of $Q^n$, then  by hypothesis $e_i(\pi,\Pp(P^n,q^n_i,i),V)\leq e_i(\pi,\Pp(P^n,q',i),V)$ for all $q'\in
\mathbb S$.  As  $q^n_i\to q^0_i$ and  $ P^n\to P^0$
we have that $\Pp(P^n,q^n_i,i)\to \Pp(P^0,q^0_i,i)$  then   we get that
$e_i(\pi,\Pp(P^0,q^0_i,i),V)\leq e_i(\pi,\Pp(P^0,q',i),V)$ for all $q'\in
\mathbb S$.

Then,
because for each $P$, $F(P)$ is a convex set and the graph $(P,F(P))$ is closed, we can apply
Kakutani's fixed point theorem, which implies the existence of a matrix $P$ that belongs to $F(P)$. Thus $P$ is a Nash minimizer of
$e(\pi, \cdot, V)$.
\end{proof}



\subsection{Uniqueness for diagonally convex costs}\label{dcc}

As shown in the previous section, if assumptions \ref{hp1} and \ref{hp2} hold,
for each $\pi$ and $V$ there exists a transition matrix $P$ which is a Nash minimizer of $e(\pi, \cdot, V)$.
In general, such minimizer may fail to be unique.
Under the diagonally convex assumption \ref{hp3} we will show uniqueness.

\begin{teo}
\label{teoA}
Suppose assumptions \ref{hp1}-\ref{hp3} hold.
Then there exists a unique transition matrix $P$ which is a Nash minimizer of $e(\pi, \cdot, V)$.
\end{teo}
\begin{proof}
Note that if $P$ is any Nash minimizer of $e(\pi, \cdot, V)$, its $i$-th line solves the constrained optimization problem
\begin{align*}
\min_q &\ e_i(\pi,\Pp(P,q,i),V)\,,\\
q_i &\geq 0 \,,\\
\sum_i q_i&=1 \,.
\end{align*}

Thus, if $P^1$ and $P^2$ are two Nash minimizers, they  satisfy the KKT conditions \cite{Pedregal}
\[
\frac{\partial e_i(\pi, P^k, V)}{\partial P_{ij}}-\nu_i^k-\theta_{ij}^k=0,    \hspace{1cm} k=1,2.
\]
\[
\sum_j P_{ij}-1=0 \,,
\]
$$ \mbox{and } \theta_{ij} P_{ij}=0\,,$$
where $\nu_i$ is the Lagrange multiplier associated with $\sum_j P_{ij}=1$ and
$\theta_{ij}\geq 0$ corresponds to the constraint $P_{ij}\geq 0$.
From the first of the three equations above  we conclude that
\[
\sum_{ij} (P^1_{ij}-P^2_{ij}) (g_{ij}(P^1) -g_{ij}(P^2)-\nu_i^1+\nu_i^2-\theta_{ij}^1+\theta_{ij}^2)=0.
\]
Therefore, using the diagonally convex property, we have
\[
\sum_{ij} (P^1_{ij}-P^2_{ij}) (-\nu_i^1+\nu_i^2-\theta_{ij}^1+\theta_{ij}^2) <0,
\]
which implies, when we use $\theta_{ij}^k P_{ij}^k=0$, that
\[
\sum_{ij} \left( -P^1_{ij} \nu^1_i  - P^2_{ij} \nu^2_i  +   P^1_{ij} \nu_i^2+P^2_{ij} \nu_i^1+P_{ij}^1 \theta_{ij}^2+P_{ij}^2\theta_{ij}^1\right)<0.
\]
Now we can use that
\[
\sum_j P^k_{ij}\nu_i^l=\nu_i^l, \;\;\;\forall \,1\leq l,k \leq 2\,,
\]
to get
\[
\sum_{ij} \left( P_{ij}^1 \theta_{ij}^2+P_{ij}^2\theta_{ij}^1\right)<0.
\]
Since $P_{ij}^1 \theta_{ij}^2, P_{ij}^2 \theta_{ij}^1\geq 0$,
we obtain a contradiction.
\end{proof}

\subsection{Uniqueness and continuity}
\label{uac}

Suppose assumption \ref{hp4} holds. Consider the map
which associates to each pair $(\pi,V)$ its unique optimizing transition matrix $P(\pi, v)$. Is it natural to ask whether
this map is a continuous function. This is addressed in the next proposition.

\begin{pro}\label{cont}
Suppose assumptions \ref{hp4}-\ref{hp5} hold.
Then $P(\pi, V)$ is a continuous function of $\pi$ and $V$.
\end{pro}
\begin{proof}
Consider sequences $\pi_n\to \pi$ and $V_n\to V$. The corresponding sequence of Nash
minimizers $P_n=P(\pi_n, V_n)$ converges to a Nash minimizer, by the continuity of $e$, assumption \ref{hp5}.
Therefore, by the uniqueness hypothesis (assumption \ref{hp4})  $P(\pi_n, V_n)\to P(\pi, V)$.
\end{proof}


\subsection{The entropy penalized model}
\label{epmd}

Now we consider a special example, the entropy penalized model.
We fix a positive constant $\ep$, and consider assume that
$c_{ij}(\pi,P_{i\cdot}) = \tilde c_{ij}(\pi)+\ep \ln(P_{ij})$, where $\tilde c_{ij}$ is a continuous function of $\pi$.
For simplicity we will drop the $\sim$.
We have
\begin{equation}
\label{lpe}
e_i( \pi, P, V)=\sum_j P_{ij}\left(c_{ij}(\pi)+\epsilon \ln P_{ij}+V_j\right).
\end{equation}
The term $\ep P_{ij}\ln P_{ij}$, with $\ep>0$, is related to entropy and forces the agents to diversify their transition choices by enforcing a penalty
if they do not do so.

It is easy to prove that there exists a unique Nash minimizing transition matrix given by
\begin{equation}\label{sol_entropy}
P_{ij}(\pi,V)=\frac{e^{-\frac{c_{ij}(\pi)+V_j}{\ep}}}{\sum_k e^{-\frac{c_{ik}(\pi)+V_k}{\ep}}}\,.
\end{equation}
Also, this transition matrix is a continuous function of $\pi$ and $V$.

Now we present a useful formula for the second derivatives of $\Gg_\pi(V)_i$:
\[
\frac{\partial^2 \Gg_\pi(V)_i}{\partial V_k\partial V_l}=\frac{p_k p_l-p_k\delta_{kl}}{\epsilon}\equiv J_{kl},
\]
where
\begin{equation}
\label{pk1}
p_k=\frac{
e^{-\frac{c_{ik}(\pi)+V_k}{\epsilon}}}{\sum_m
e^{-\frac{c_{im}(\pi)+V_m}{\epsilon}}}.
\end{equation}
Because $\sum_{k} J_{kl}=0$, the matrix $J_{kl}$ has a zero eigenvalue, which is a reflection of the fact that $\Gg_\pi$ commutes with addition of
constants. As we show in the next proposition this eigenvalue is simple.
\begin{pro}
\label{simpleev}
Suppose $0<p_k<1$, $\sum_k p_k=1$ and let \[J_{kl}=\frac{p_k p_l-p_k\delta_{kl}}{\epsilon}.\]
Then $0$ has simple multiplicity.
\end{pro}
\begin{proof}
Observe that
\[
J=\frac 1 \epsilon D (Q-I),
\]
where
\[
D=\diag\{p_1, \hdots, p_n\}
\]
and
\[
Q_{kl}=p_l
\]
If $0$ is not a simple eigenvalue, it would mean that there exists $w$ and $v$ which are linearly independent eigenvectors corresponding to this
eingenvalue. But then $v$ and $w$ are eigenvectors of $Q$ corresponding to the eigenvalue $1$. But this contradicts Perron-Frobenius theorem
because the eigenvalue $1$ is a simple eigenvalue of $Q$.
\end{proof}



\section{Properties of $\Gg$}
\label{gprop}

In this section we discuss the main properties of the operator $\Gg$. In \S \ref{gmonot} we show that assumption \ref{hp9} is a consequence of assumptions \ref{hp4} and \ref{hp7}, and
in \S\ref{gconcav} we study concavity properties of  $\Gg$. A-priori bounds, essencial to establishing existence of stationary solutions
are considered in \S\ref{apb} and, finally, in \S\ref{scepm} we prove strict concavity of $\Gg$ for the entropy penalized model with two states.

\subsection{Assumption \ref{hp9}}
\label{gmonot}

\begin{pro}
\label{monotoneprop}
Suppose assumptions \ref{hp4} and \ref{hp7} hold. Then assumption \ref{hp9} holds.
\end{pro}
\begin{proof}
Let $V^k\in \Rr^d$, $k=1,2$. Let $i\in \argmax V^1-V^2$. By adding a constant, we may assume that $V^1_i=V^2_i$, and so because $\Gg$ commutes
with the addition of constants, it suffices to check that
\[
\Gg_\pi(V^1)_i\leq \Gg_\pi(V^2)_i.
\]
Because $i$ is a maximizer of $V^1_j-V^2_j$, for all $j$ we have
$V^1_j-V^2_j\leq 0$, that is $V^1_j\leq V^2_j$.
Let $P_{ij}^2$ be such that
\[
\Gg_\pi(V^2)_i=\sum_j c_{ij}(\pi, P_{i\cdot}^2) P_{ij}^2+ V^2_k P_{ij}^2.
\]
Then
\[
\Gg_\pi(V^1)_i\leq \sum_j c_{ij}(\pi, P_{i\cdot}^2) P_{ij}^2+ V^1_j P_{ij}^2\leq \sum_j c_{ij}(\pi, P_{i\cdot}^2) P_{ij}^2+ V^2_j P_{ij}^2
=\Gg_\pi(V^2)_i.
\]
Arguing similarly we obtain the opposite inequality when $i\in \argmin V^1-V^2$.
\end{proof}

\subsection{Concavity}
\label{gconcav}

If assumption \ref{hp7} holds, for each fixed index $i$ the mapping
\[
V\mapsto \Gg_{\pi}(V)_i
\]
is concave since
\begin{equation}
\label{mnq}
\Gg_{\pi}(V)_i=  \min_{P_{i\cdot}\in \mathbb{S}} \sum_j c_{ij}(\pi, P_{i\cdot})P_{ij}+
V_j P_{ij},
\end{equation}
is a pointwise minimum of linear functions of $V$.
Furthermore, since $\pi\geq 0$,
\[
V\mapsto \pi\cdot \Gg_{\pi}(V)
\]
is also concave.

Suppose that
$P(\pi,V)$,  the transition matrix that realizes the minimum in \eqref{mnq}, 
 is differentiable with respect to $V$. We will use the notation $P^{\pi, V}=P(\pi, V)$.
 Then
\[
\frac{\partial \Gg_{\pi}(V)_i}{\partial V_j}=P_{ij}^{\pi, V}.
\]
In general, however, $P^{\pi, V}$ is may not be differentiable.
Nevertheless,
we have the following rigorous statement:
\begin{pro}\label{concav}
Suppose assumptions \ref{hp4} and \ref{hp7} hold.
Then
\begin{itemize}
\item[(a)] If $V^1$ and $V^2$ are any two given vectors, we have
\begin{equation}
\label{kav2}
\Gg_\pi(V^2)_i\leq \Gg_\pi(V^1)_i+\sum_jP_{ij}^{V^1,\pi} (V^2-V^1)_j.
\end{equation}
\item[(b)] If $V^2 \geq V^1$, then
\[
0 \leq \Gg_\pi(V^2)_i- \Gg_\pi(V^1)_i \leq \sum_jP_{ij}^{V^1,\pi} (V^2-V^1)_j.
\]
\end{itemize}
\end{pro}
\begin{proof}
(a) Observe that
\[
\Gg_\pi(V^2)_i\leq \sum_j c_{ij}(\pi, P_{i\cdot}^{V^1,\pi})P_{ij}^{V^1,\pi}+ V^2_j P_{ij}^{V^1,\pi}=\Gg_\pi(V^1)_i+\sum_j P_{ij}^{V^1,\pi} ( V^2-V^1)_j.
\]
(b) We just need to prove the first inequality. For this, note that
\begin{align*}
\Gg_\pi(V^2)_i & = \sum_j P_{ij}^{V^2,\pi}(V_j^2+c_{ij}(\pi,P_{i\cdot}^{V_2,\pi}))  \\
&\geq
\sum_j P_{ij}^{V^2,\pi}(V_j^1+c_{ij}(\pi,P_{i\cdot}^{V_2,\pi}))
\geq \Gg_\pi(V^1)_i.
\end{align*}
\end{proof}

Note that, by multiplying \eqref{kav2} by $\pi_i$ and adding, we obtain
\begin{equation}
\label{kav3}
\pi\cdot \left(\Gg_\pi(V^2)-\Gg_\pi(V^1)\right) - (V^2-V^1)\cdot \Kk_{V^1}(\pi)\leq 0.
\end{equation}
%
%
Since $\Gg$ commutes with the addition of constants it is not possible to establish a strict concavity estimate like
\[
\pi\cdot (\Gg_{\pi}(V^2)-\Gg_{\pi}(V^1))+ (V^1-V^2)\cdot \Kk_{V^1}(\pi)\leq -\gamma_{\pi} \|V^1-V^2\|^2
\]
However, in certain cases it is possible to obtain the following strict concavity estimate:
\[
\pi \cdot (\Gg_{\pi}(V^2)-\Gg_{\pi}(V^1))+(V^1-V^2)\cdot \Kk_{V^1}(\pi) \leq -\gamma_{\pi} \|V^1-V^2\|^2_\#.
\]
In \S \ref{scepm} we will give an explicit example.

\subsection{A-priori bounds}
\label{apb}

In the next proposition we give some a-priori bounds which are essencial to establishing the existence of fixed points (theorem \ref{teoB}).

\begin{pro}
\label{prop8}
Suppose assumption \ref{hp6} holds.
Then
for any $(\pi,V)\in \mathbb{S}\times \Rr^d$ and all indices
$i, i'$ we have
\begin{equation}
\label{compest1}
|\Gg_\pi(V)_i-\Gg_\pi(V)_{i'}|\leq C.
\end{equation}
\end{pro}


\begin{proof}
Fix $\pi$ and $V$ let $(\Gg_\pi, \Kk_V)$ be as in \eqref{ggdef} and \eqref{kkdef}.
Let  $P$ be  the optimal transition matrix such that
\[
\mathcal G_{\pi}(V)_i=   \sum_j c_{ij}(\pi, P) P_{ij}+
V_j P_{ij}.
\]
Since replacing the $i'$-th line by the $i$-th line in $P$ yields a sub-optimal choice, we have
\[
\mathcal G_{\pi}(V)_{i'}\leq   \sum_j c_{i'j}(\pi,\rho_{i,i'}(P))P_{ij}+
V_j P_{ij}.
\]
Hence
$$\mathcal G_{\pi}(V)_{i'}-\mathcal G_{\pi}(V)_i\leq   \sum_j
[-c_{ij}(\pi,P)+c_{i'j}(\pi,\rho_{i,i'}(P))]P_{ij}.  $$
If we exchange the role of $i$ and $i'$ we have the desired estimate.
\end{proof}

\subsection{Strict concavity for the entropy penalized model}
\label{scepm}


To show that the entropy penalized model satisfies the strict concavity property of assumption \ref{hp10}, it suffices
to show that the restriction of the linear form given by the matrix $D^2_{V^2} \Gg_{\pi}(V)_i$ to the space of
vectors $X\in \Rr^d$ with $\sum_k X_k=0$ is uniformly definite positive for each $i$. This holds because of proposition
\ref{simpleev}, which states that the eigenvalue $0$ is simple, and the corresponding eigenvector is $Y=(1, \hdots, 1)$.
The uniformity follows
from the a-priori bounds in the previous section, which allow us to assume that $\|V\|_\#$ is bounded. In fact (uniform) strict concavity only holds
for bounded $\|V\|_\#$. However, for the purposes of this paper this is enough because of the a-priori bounds in \S \ref{apb}. 
Thus if $c_{ij}$ is bounded and $\|V\|_\#$ is bounded we have $0<p_k<1$, and so assumption \ref{hp10} holds.





\section{Stationary Solutions}
\label{statsol}

In this section we study stationary solutions to mean field games.
After the characterization of the critical value, in \S \ref{rcv}, as the average cost for the population of agents,
we address the question of uniqueness of the critical
value $\bar \lambda$ for which \eqref{stat} admits a solution. In \S \ref{4p1} we give an example where $\bar \lambda$ is non-unique.
However, after addressing the issue of existence of stationary solutions (in \S \ref{essln}),
we revisit the uniqueness problem
in \S\ref{uniq} giving conditions which imply uniqueness $\bar \lambda$, $\bar \pi$ and $\bar V$.
These conditions are variations of the monotonicity conditions in \cite{ll1}.
The entropy penalized model is revisited
in \S\ref{epss}, and the large entropy limit is considered in \S\ref{lelmt}, where we establish uniqueness of stationary solution (proposition
\ref{mainproA}).
This uniqueness proof uses a strong contraction argument and is thus suitable for the numerical approximation of large entropy penalized mean
field games.
We end this section, in \S\ref{oss2},
with a discussion of optimal stationary solutions, where certain variational problems give rise to mean field games (see \cite{ll3}, \S 2.6,
for related problems).

\subsection{Representation of the critical value}
\label{rcv}

We will now give a representation formula for the critical value as the
average transition cost.

\begin{pro}
Suppose assumption \ref{hp4} holds.
Let $(\bar \pi, \bar V)$ be a stationary solution to \eqref{stat}, and $\bar \lambda$ the
corresponding critical value.
Let $\bar P$ be the optimal transition matrix.
Then
\[
\bar \lambda= \sum_{ij} \pi_i c_{ij}(\bar \pi,\bar P) \bar P_{ij}.
\]
\end{pro}
\begin{proof}
For each $1 \leq i \leq d$,
\begin{equation}\label{la}
\bar \lambda
=\sum_j c_{ij}(\bar \pi,\bar P)\bar P_{ij}-\left( \bar V_i  -\sum_j \bar P_{ij} \bar V_j \right).
\end{equation}
%
Note that $\sum_j c_{ij}(\bar \pi,\bar P)\bar P_{ij}$ can be seen as the expected cost of transition agents that are in state $i$ will have when moving to other states.
If we multiply \eqref{la} by $\bar \pi_i$ and add, for $1 \leq i \leq d$, we get
\[
\bar \lambda=\sum_{i,j} c_{ij}(\bar \pi,\bar P)\bar \pi_i \bar P_{ij}- \sum_{i,j}  \left(\bar V_i  - \bar V_j \right)
\bar \pi_i \bar P_{ij}.
\]
Let
$\mu^{\bar \pi \bar P}$ denote the probability measure on the set $\{1,2,...,d\}^2$ given by
$\mu^{\bar \pi \bar P}_{ij}=\pi_i \bar P_{ij}$.
Since
$\bar \pi =\bar \pi \bar P$, we have
\[\sum_i\mu^{\bar \pi \bar P}_{ij}=\sum_i\mu^{\bar \pi \bar P}_{ji}=\pi_j.
 \]
Therefore
 \[
 \sum_{i,j}  \left(\bar V_i  - \bar V_j \right)
\bar \pi_i \bar P_{ij}=\sum_{i,j}  \left(\bar V_i  - \bar V_j \right) \mu^{\bar \pi \bar P}_{ij}=0.
 \]
 So
 \[
 \bar \lambda=
\sum c_{ij}(\bar \pi,\bar P)\bar \mu^{\bar \pi \bar P}_{ij}.
\]
\end{proof}


\subsection{Non-uniqueness of the critical value}
\label{4p1}

In this section we show that the critical value may not be unique. Consider the following example: $c_{ij}(\pi)$ given by
\[
c_{12}=c_{21}=100,
\]
and
\[
c_{11}(\pi^\theta)=c_{22}(\pi^\theta)=\theta.
\]
where $\pi^\theta=(\theta, 1-\theta)$, for $0\leq \theta\leq 1$. Then $V^\theta=(0,0)$, $\lambda^\theta=\theta$, $\pi^\theta$ and
\[
P^\theta=
\left[
\begin{array}{cc}
1&0\\
0&1
\end{array}
\right]
\]
is a stationary solution.

\subsection{Existence of stationary solutions}
\label{essln}


\begin{teo}
\label{teoB}
Suppose that assumptions \ref{hp4}-\ref{hp5} hold.
Assume further that there exists $C$ such that for any $(\pi,V)\in \mathbb{S}\times \Rr^d$ and all indices
$i, i'$ we have \eqref{compest1}.
Then there exists a pair of vectors $(\bar \pi, \bar V)$, a constant $\bar \lambda$ and a transition matrix $\bar P$
such that for all $i$,
\[
\Gg_{\bar \pi}(\bar V)_i= \sum_j c_{ij}(\bar\pi,\bar  P)\bar P_{ij}+\bar V_j\bar P_{ij}=\bar V_i+\bar \lambda,
\]
and $\bar \pi =\bar \pi \bar P$.
\end{teo}

Note that \eqref{compest1} will hold, by proposition \ref{prop8}, if we assume additionally assumption \ref{hp6}.

\begin{proof}
Since \ref{hp4}-\ref{hp5} hold, by proposition \ref{cont}, the optimal transition matrix $P(\pi, V)$ is a continuous function.
Therefore the operator $\mathcal G_{\pi}(\cdot):\Rr^d/\Rr\to \Rr^d/\Rr$ is continuous. Furthermore,
by estimate
\eqref{compest1}, $\|\Gg_{\pi}(V)\|_{\#}$
is uniformly bounded for any $V\in \Rr^d$.

Consider the mapping   $(\Gg_\pi(V), \Kk_V(\pi)):(\Rr^d/\Rr)\times \mathbb{S}\to (\Rr^d/\Rr)\times \mathbb{S}$.
By Brouwer's fixed point theorem this mapping has a fixed point, which is a stationary solution since
\[
\begin{cases}
\Gg_{\bar \pi}(\bar V)=\bar \lambda+\bar V\\
\Kk_{\bar V}(\bar \pi)=\bar \pi.
\end{cases}
\]
\end{proof}

\subsection{Uniqueness of stationary solutions}\label{uniq}

Now we address the problem of uniqueness of stationary solutions.
The results in this section use the monotonicity methods introduced in \cite{ll1, ll2, ll3} - in this
discrete setting, different versions of the hypothesis will yield several uniqueness results.
Under assumption \ref{hp8} we will prove, in proposition
\ref{ups}, the uniqueness of stationary distribution $\pi$. The uniqueness of critical value is established in proposition \ref{ucv},
using assumption \ref{hp9}, and finally under assumption \ref{hp10} we obtain the uniqueness of $\pi$, $V$ and $\lambda$
in theorem \ref{teoC}.

%



\begin{pro}\label{ups}
Suppose assumptions \ref{hp4}, \ref{hp7}, and \ref{hp8} hold. Let $(\pi^k, V^k)$ , $k=1,2$, be stationary solutions:
\[
\Kk_{V^k}(\pi^k)=\pi^k, \qquad \Gg_{\pi^k}(V^k)=\lambda^k+V^k\,
\]
where $\lambda^k$ are constants.
 Then $\pi^1=\pi^2$.
\end{pro}
\begin{proof}
From the hypothesis we have
\begin{align*}
0= &(V^1-V^2)\cdot (\Kk_{V^1}(\pi^1)-\pi^1-\Kk_{V_2}(\pi^2)+\pi^2)+\\&
+(\pi^1-\pi^2)((\Gg_{\pi^2}(V^2) -V^2)-(\Gg_{\pi^1}(V^1) -V^1))
+(\lambda^1-\lambda^2)\sum_{i}(\pi^1_i-\pi^2_i).
\end{align*}
Note that the last term vanishes since $\sum_{i}(\pi^1_i-\pi^2_i)=0$.
Rewriting we have
\begin{align*}
0=&\pi^1 \cdot (\Gg_{\pi^1}(V^2)-\Gg_{\pi^1}(V^1))+(V^1-V^2)\cdot \Kk_{V^1}(\pi^1) +\\&
+\pi^2 \cdot (\Gg_{\pi^2}(V^1)-\Gg_{\pi^2}(V^2))+(V^2-V^1)\cdot \Kk_{V^2}(\pi^2) +\\&
+\pi^1\cdot (\Gg_{\pi^2}(V^2)-\Gg_{\pi^1}(V^2))
+\pi^2\cdot (\Gg_{\pi^1}(V^1)-\Gg_{\pi^2}(V^1)).
\end{align*}
By proposition \ref{concav} and (\ref{kav3}), the first term above satisfies
\[
\pi^1\cdot(\Gg_{\pi^1}(V^2)-\Gg_{\pi^1}(V^1))+(V^1-V^2)\cdot \Kk_{V^1}(\pi^1) \leq 0,
\]
and similarly for the second term.

To analize the third term, observe that, by
assumption \ref{hp8}, we have
\[
\pi^1\cdot (\Gg_{\pi^2}(V^2)-\Gg_{\pi^1}(V^2))  +\pi^2\cdot (\Gg_{\pi^1}(V^1)-\Gg_{\pi^2}(V^1))  \leq -\gamma \|\pi^1- \pi^2\|^2.
\]
Therefore, the estimates above imply
$
\gamma \|\pi^1- \pi^2\|^2\leq 0
$.
\end{proof}

Now we establish the uniqueness of the critical value:
\begin{pro}
\label{ucv}
Suppose assumptions \ref{hp4}, \ref{hp7} and \ref{hp9} hold.
Let $\pi\in \Sss$, $\lambda^k\in \Rr$ and $V^k\in \Rr^d$, $k=1,2$ be solutions of
\[
\Gg_\pi(V^k)=\lambda^k+ V^k.
\]
Then $\lambda^1=\lambda^2$.
\end{pro}
\begin{proof}
Choose $i\in \argmax V^1-V^2$. Then
\[
\lambda^1=\Gg_{\pi}(V^1)_i-V^1_i\leq \Gg_{\pi}(V^2)_i-V^2_i=\lambda^2.
\]
By choosing $i\in \argmin V^1-V^2$ we obtain the opposite inequality,  which then implies $\lambda^1=\lambda^2$.
\end{proof}

Therefore, under assumptions \ref{hp8} and \ref{hp9}, we have both uniqueness of the stationary distribution $\pi$ and critical value $\lambda$.
We now address the uniqueness of the stationary value function $V$.

\begin{teo}\label{teoC}
Suppose assumptions \ref{hp4}, \ref{hp7}, \ref{hp8} and \ref{hp10} hold.
Let $(\pi^k, V^k)$, $k=1,2$, be stationary solutions:
\[
\Kk_{V^k}(\pi^k)=\pi^k, \qquad \Gg_{\pi^k}(V^k)=\lambda^k+V^k\,
\]
where $\lambda^k$ are constants.
Then
\begin{itemize}
\item[(a)] $\pi^1=\pi^2$,
\item[(b)] $V^2=V^1+k$, where $k$ is a constant vector,
\item[(c)] $\lambda_1=\lambda_2$.
\end{itemize}
\end{teo}
\begin{proof}
If we follow the proof of Proposition \ref{ups}, we can use assumption \ref{hp10} to get
\[0= -\gamma_{1} \|V^1-V^2\|_{\#}^2 -\gamma_{2} \|V^1-V^2\|_{\#}^2 -\gamma \|\pi^1- \pi^2\|^2 \,.\]
This implies items (a) and (b).
 To get item (c), we observe that
 \[ V^2+\lambda^2= \Gg_\pi(V^2)= \Gg_\pi(V^1+k) = \Gg_\pi(V^1)+k = V^1+\lambda^1+k = V^2+\lambda^1 \,, \]
where $\pi = \pi_1=\pi_2$, in the first and fourth equalities we used $\Gg_{\pi}(V^k)=V^k+\lambda^k$\,, in the second and fifth we used
  item (b), 
   and in the third we used the fact that $\Gg_\pi$ commutes with constants.
\end{proof}

\subsection{Entropy penalized stationary solutions}\label{epss}

Now we consider
the entropy penalized model
We will present a simple proof of existence of solutions that relies on the special structure
of the problem.
A simple computation yields
\[
\Gg_{\pi}(V)_i=
-\ep \ln \left[\sum_{k}e^{-\frac{c_{ik}(\pi)+V_k}{\ep}   }   \right].
\]
We will suppose further:
\begin{hypothesis}
\label{hpep1}
The function $c_{ij}(\pi)$ is continuous.
\end{hypothesis}

\begin{pro}\label{sep}
Suppose assumption \ref{hpep1} holds.
Consider the entropy penalized model \eqref{lpe}. Then
there exists a pair of vectors $(\bar \pi, \bar V)\in \Sss\times \Rr^d$, a constant $\bar \lambda\in \Rr$
such that
\begin{align*}
\Gg_{\bar \pi}(\bar V)&=
\bar \lambda
+\bar V
\end{align*}
and $\bar \pi =\bar \pi \bar P=\Kk_{\bar V}(\bar \pi)$.
\end{pro}
\begin{proof}
%
%
%
Define the strictly positive linear operator that associates to each vector $\psi\in \Rr^d$ the vector
 $$\mathcal L_{\pi}(\psi)_i=\sum_{k}e^{-\frac{c_{ik}(\pi)}{\ep}  }\psi_k.    $$
Let $e^{-\frac{\lambda_{\pi}}{\ep}}$ be the largest eigenvalue of the operator $\mathcal L_{\pi}(\psi)$
and $\psi^{\pi}$ the unique normalized eigenvector associated to $e^{-\frac{\lambda_{\pi}}{\ep}}$, i.e.,
  $$
\mathcal L_{\pi}(\psi^{\pi})= e^{-\frac{\lambda_{\pi}}{\ep}}	 \psi^{\pi}.$$
By Perron-Frobenius Theorem, $\psi_{\pi}$ is a strictly positive vector which is
a continuous function of $\pi$.
We  can define $V^{\pi}$ as $ \psi^{\pi}_k=e^{-\frac{V^{\pi}_k}{\ep}}$.
Let  $\mathcal E[\phi]_j=e^{-\frac{\phi_j}{\ep}}$ be the exponential transformation.
These operators are related by
$$\mathcal L_{\pi}\circ \mathcal E= \mathcal E\circ \mathcal G_{\pi} \,.$$
Hence
 \begin{equation}\label{G}
 \mathcal G_{\pi} ( V_{\pi})= V_{\pi} + \lambda_{\pi}  \,.
\end{equation}

Define a new probability vector
$$ \Kk(\pi)_j=\sum_i\pi_iP_{ij}(\pi,V)=\sum_i\pi_i \frac{e^{-\frac{c_{ij}(\pi)+V^{\pi}_j}{\ep}}}{\sum_k e^{-\frac{c_{ik}(\pi)+V^{\pi}_k}{\ep}}} \,  $$
Thus we have defined a operator $\mathcal K:\mathbb S\to\mathbb S$ which is continuous.
By Brower's fixed point theorem, $\mathcal K$ has a fixed point $\bar\pi$.
Define $\bar V^{\bar\pi}$ and $\bar \lambda_{\bar\pi}$ as above, and $\bar P_{ij}=\frac{e^{-\frac{c_{ij}(\bar \pi)+\bar V^{\bar\pi}_j}{\ep}}}{\sum_k e^{-\frac{c_{ik}(\bar \pi)+\bar V^{\bar\pi}_k}{\ep}}}$, then (\ref{G}) holds.
\end{proof}


\subsection{Stationary Solutions with large entropy}
\label{lelmt}

The construction of fixed points for the entropy penalized model in
the last section depends on Brower's fixed point theorem. In the case of large entropy we can use a contraction argument to establish
the existence of a stationary solution. Before proving and stating this result we need an elementary lemma
\begin{lem}
\label{techlem}
Let $T:\Rr^n\times \Rr^n\to \Rr^n\times \Rr^n$ be a $C^1$ mapping. Suppose
\[
DT=
\left[
\begin{array}{cc}
E_1&M\\
E_2&E_3
\end{array}
\right].
\]
where $E_k$, $M$ are $n\times n$ matrices. If $\|M\|$ is bounded and $\|E_k\|$ is sufficiently small then
$T^2$ is a strong contraction in $\Sss^d \times \Rr^d/\Rr$.
\end{lem}
\begin{proof}
It suffices to observe that
\[
D(T^2)(x)=DT(T(x)x) DT(x)
\]
can be written as
\[
D(T^2)=
\left[
\begin{array}{cc}
\tilde E_1&\tilde M\\
\tilde E_2&\tilde E_3
\end{array}
\right]
\left[
\begin{array}{cc}
E_1&M\\
E_2&E_3
\end{array}
\right]
=
\left[
\begin{array}{cc}
\tilde E_1 E_1+\tilde M E_2  &\tilde E_1 M+\tilde M E_3\\
\tilde E_2 E_1+\tilde E_3 E_2&\tilde E_2 M+\tilde E_3 E_3
\end{array}
\right],
\]
and therefore $\|D(T^2)\|<1$.
\end{proof}

To establish the main result in this section we need to replace assumption \ref{hpep1} by
\begin{hypothesis}
\label{hpep2}
The function $c_{ij}(\pi)$ is a $C^1$ function.
\end{hypothesis}

\begin{pro}
\label{mainproA}
Suppose assumption \ref{hpep2} holds.
%
Then, for large $\ep$, there is a unique stationary solution. Additionally, let
$T(\pi,V)=( \mathcal K_{V,\ep}(\pi),\Gg_{\pi}(V))$, then $T^2$
is a strong contraction.
\end{pro}
\begin{proof} 
From proposition \ref{prop8} we have $\|\Gg_{\pi}(V)\|_\#$ is uniformly bounded. Therefore it suffices to show that the operator $(\Kk, \Gg)$ is a strong contraction
for $V$ in a compact set (with respect to the norm $\|\cdot\|_\#$).
Also, we can replace $\Gg$ by
\[
\hat \Gg_{\pi}(V)_i=\Gg_{\pi}(V)_i-\frac 1 d\sum_k \Gg_{\pi}(V)_k,
\]
and $\Kk_V(\pi)$ by
\[
\hat \Kk_V(\pi)_i=\Kk_V(\pi)_i-\frac 1 d \left(\sum_j \Kk_V(\pi)_j -1\right).
\]
Since any fixed point to $(\hat \Gg, \hat \Kk)$ is a stationary solution.
In this way $\|\Gg_{\pi}(V)\|_{\#}=\|\hat \Gg_{\pi}(V)\|$, and for any $\pi$ (not necessarily a probability measure),
$\hat \Kk_V(\pi)$ is a probability measure, and agrees with $\Kk_V(\pi)$ if $\pi$ is a probability measure.

We will show that, for $\epsilon$ sufficiently large, the pair $(\hat \Gg, \hat \Kk)$
satisfies the hypothesis of lemma \ref{techlem}.
To do so, we first compute the matrix
\[
\left[
\begin{array}{cc}
\frac{\partial \Gg}{\partial V}&\frac{\partial \Gg}{\partial \pi}\\
\frac{\partial \Kk}{\partial V}&\frac{\partial \Kk}{\partial \pi}
\end{array}
\right].
\]
We will show that, when $\epsilon \to \infty$, this matrix converges to
\[
\left[
\begin{array}{cc}
[1/d]&[\frac 1 d \sum_k \frac{\partial c_{ik}}{\partial \pi_j}]\\
\left[0\right]&[1/d]
\end{array}
\right],
\]
where we denote by $[a]$ the $d\times d$ matrix whose entries are all identical to $a$.

Since $\Gg_{ \pi,\ep}( V)_i=-\ep \ln\left(\sum_k e^{-\frac{c_{ik}( \pi)+ V_k}{\ep}}\right)$, we have:
\[
\left(\frac{\partial \Gg_{\pi} (V)}{\partial V_j} \right)_i =
\frac{
e^{-\frac{c_{ij}( \pi)+ V_j}{\ep}}
}{
\sum_k e^{-\frac{c_{ik}( \pi)+ V_k}{\ep}}
} \longrightarrow 1/d,
\]
when $\ep \to \infty$. We also have
\[
\left(\frac{\partial \Gg_{\pi} (V)}{\partial \pi_j} \right)_i =
\frac{\sum_k     \frac{\partial c_{ik}(\pi)}{\partial \pi_j}   e^{-\frac{c_{ij}( \pi)+ V_j}{\ep}}
       }
       {
       \sum_k e^{-\frac{c_{ik}( \pi)+ V_k}{\ep}}
       } \longrightarrow \frac 1 d\sum_k \frac{\partial c_{ik}(\pi)}{\partial \pi_j}  ,
\]
when $\ep \to \infty$.

Now we consider
$
\mathcal K_{V}(\pi)_i = \sum_k \pi_k P_{ki}$, where
\[
P_{ki}=\frac{e^{-\frac{c_{ki}(\pi)+V_i}{\ep}}}{\sum_l e^{-\frac{c_{kl}(\pi)+V_l}{\ep}}} \,.
\]
We have
\[
\left(\frac{\partial \mathcal K_{V}(\pi)}{\partial \pi_j} \right)_i =
P_{ji} +
\sum_k \pi_k
\frac
{
\sum_l e^{-\frac{c_{kl}(\pi)+V_l}{\ep}}
e^{-\frac{c_{ki}(\pi)+V_i}{\ep}}
\left(
\frac{\partial c_{kl}(\pi)}{\partial \pi_j}
-\frac{\partial c_{ki}(\pi)}{\partial \pi_j}
\right)
}
{\ep \left(\sum_l e^{-\frac{c_{kl}(\pi)+V_l}{\ep}}\right)^2}
\]
If we take $\ep \to \infty$, the second term tends to zero while $P_{ij} \to \frac{1}{d}$. Thus
\[
\left(\frac{\partial \mathcal K_{\bar V,\ep}(\bar \pi)}{\partial \pi_j} \right)_i \longrightarrow \frac{1}{d}\,. \]
Since
\[
\left(\frac{\partial \mathcal K_{\bar V,\ep}(\bar \pi)}{\partial V_j} \right)_i
=
\sum_k \pi_k
\frac
{
\sum_l e^{-\frac{c_{kl}(\pi)+V_l}{\ep}}
e^{-\frac{c_{ki}(\pi)+V_i}{\ep}}
\left(
\delta_{lj}
-\delta_{kj}
\right)
}
{\ep \left(\sum_l e^{-\frac{c_{kl}(\pi)+V_l}{\ep}}\right)^2},
\]
we obtain
\[
\left(\frac{\partial \mathcal K_{\bar V,\ep}(\bar \pi)}{\partial V_j} \right)_i \longrightarrow 0, \] when $\ep \to \infty$.

Finally, observe that
\[
\left[
\begin{array}{cc}
\frac{\partial \hat\Gg}{\partial V}&\frac{\partial \hat \Gg}{\partial \pi}\\
\frac{\partial \hat \Kk}{\partial V}&\frac{\partial \hat \Kk}{\partial \pi}
\end{array}
\right]
=
\left[
\begin{array}{cc}
I-[1/d]&\left[0\right]\\
\left[0\right]&I-[1/d]
\end{array}
\right]
\left[
\begin{array}{cc}
\frac{\partial \Gg}{\partial V}&\frac{\partial \Gg}{\partial \pi}\\
\frac{\partial \Kk}{\partial V}&\frac{\partial \Kk}{\partial \pi}
\end{array}
\right].
\]
This shows that, in the limit $\epsilon\to \infty$, we have
\[
\left[
\begin{array}{cc}
\frac{\partial \hat\Gg}{\partial V}&\frac{\partial \hat \Gg}{\partial \pi}\\
\frac{\partial \hat \Kk}{\partial V}&\frac{\partial \hat \Kk}{\partial \pi}
\end{array}
\right]
\to
\left[
\begin{array}{cc}
\left[0\right]&M\\
\left[0\right]&[0]
\end{array}
\right],
\]
where
\[
M=(I-[1/d])[\frac 1 d \sum_k \frac{\partial c_{ik}}{\partial \pi_j}].
\]
Thus, for $\epsilon$ large enough lemma \ref{techlem} yields the strong contraction property of $T^2$.
\end{proof}



\subsection{Optimal Stationary Solutions}
\label{oss2}

Given a probability measure $\eta_{ij}$, $1\leq i, j\leq d$, define $\pi_i^\eta=\sum_j \eta_{ij}$ and
let $P_{ij}^\eta$ be a stochastic matrix such that $\eta_{ij}=\pi_i^\eta P_{ij}^\eta$. If $\pi^\eta$ never vanishes
then $P^\eta$ is uniquely defined by $P_{ij}^\eta=\frac{\eta_{ij}}{\pi_i^\eta}$. A probability measure $\eta_{ij}$ is stationary
if
\begin{equation}
\label{holonomy}
\sum_j \eta_{ij}=\sum_j \eta_{ji}.
\end{equation}
For our purposes, in this section it is convenient to consider the following two auxiliary assumptions:
\begin{hypothesis}
\label{hpauxa}
For each $1\leq i\leq d$ the mapping $P_{i\cdot}\mapsto \sum_j c_{ij}(P_{i\cdot})P_{ij}$ is convex.
\end{hypothesis}
\noindent and
\begin{hypothesis}
\label{hpauxb}
The mapping
$\eta\mapsto \sum_{i,j} \pi^\eta c_{ij}(P_{i\cdot}^\eta) P^\eta_{ij}$ is strictly convex.
\end{hypothesis}
Consider a $C^1$ convex function function $f:\Rr^d\to \Rr$.
Consider the problem
\begin{equation}
\label{nloss}
\min_\eta \sum_{ij} \pi_i^\eta c_{ij}(P^\eta) P_{ij}^\eta+f(\pi^\eta),
\end{equation}
where the minimum is taken over all probability measures $\eta$ satisfying \eqref{holonomy}.

\begin{pro}
\label{mainproB}
Let $\eta>0$ be a solution of \eqref{nloss}. Let $V^\eta\in \Rr^d$ be the Lagrange multiplier associated to
the constraint \eqref{holonomy}, and $\lambda^\eta$ the Lagrange multiplier corresponding to $\sum_{ij} \eta_{ij}=1$.
Then $(\pi^\eta, V^\eta)$ is a stationary solution of
\[
V^\eta+\lambda^\eta=\Gg_{\pi^\eta}(V^\eta),
\]
where
\[
\Gg_\pi(V)_i=\frac{\partial f}{\partial \pi_i}(\pi)+\min_{P_{i\cdot}} \sum_j c_{ij}(P_{i\cdot})P_{ij}+ P_{ij} V_j.
\]
\end{pro}
\begin{proof}
Let $\eta$ be as in the statement. As before write $\eta_{ij}=\pi_i^\eta P_{ij}^\eta$.
Then both $\pi^\eta$ and $P_{ij}^\eta$ are critical points of the functional
\begin{equation}
\label{fctnal}
f(\pi)+\sum_{i} \pi_i \left[\left(\sum_j c_{ij}(P) P_{ij}+V_j^\eta P_{ij}\right)-V_i^\eta-\lambda^\eta\right].
\end{equation}
Consequently, for each $i$, $P_{i\cdot}$ is a critical point of
\[
\sum_j c_{ij}(P) P_{ij}+V_j^\eta P_{ij},
\]
which by the convexity hypothesis, assumption \ref{hpauxa}, is a minimizer.
Furthermore, by differentiating \eqref{fctnal} with respect to $\pi_i$, we  obtain
\[
V_i+\lambda^\eta=\Gg_{\pi^\eta}(V^\eta)_i.
\]
Finally, we can write \eqref{holonomy} as
\[
\pi^\eta=\Kk_{V^\eta}(\pi^\eta),
\]
which ends the proof.
\end{proof}

\begin{pro}
Suppose assumptions \ref{hp7},  and \ref{hpauxb} hold. Then there exists at most one stationary solution
with $\pi, P>0$.
\end{pro}
\begin{proof}
This proposition follows from the well known fact (see for instance \cite{Pedregal})
 that for strictly convex objective functions under linear constraints the KKT conditions
are not only necessary but also sufficient.
\end{proof}

Note that if $f$ is a strictly convex function, then the previous proposition gives us another proof of the existence and uniqueness of the stationary solution in the case  $c_{ij}(\pi,P)=\tilde c_{ij}(P_{i\cdot})+W(\pi)$, with $W(\pi)=\frac{\partial f}{\partial \pi_i}(\pi)$.

\section{Solutions to the Mean Field Game initial-terminal value problem}
\label{esmfg}


In this section we prove the existence (\S\ref{esitp}) and uniqueness (\S\ref{usitvp})
of solutions to the initial-terminal value problem.

\subsection{Existence of Solutions}
\label{esitp}

\begin{teo}\label{teoD}
Suppose assumptions \ref{hp4} and \ref{hp5} hold. Then
for any initial probability vector $\tilde \pi\in \Sss$ and terminal cost $\tilde V$ there exists a solution
$$\{(\pi^n,V^n) \;  ; \;0\leq n \leq N \}$$ to the initial-terminal value problem for the mean field game
with $\pi^0=\tilde \pi$ and $V^N=\tilde V$.
\end{teo}
\begin{proof}
Suppose we are given a sequence $\pi^{n,0}\in \Sss^{N+1}$ of probability vectors, with $\pi^{0,0}=\tilde \pi$.
Define, for $0\leq n\leq N$
\[
V^{n,0}=\Gg_{\pi^{n,0}}(V^{n+1,0}),
\]
with $V^{N,0}=\tilde V$. Then let
\[
\pi^{n+1,1}=\Kk_{V^{n,0}}(\pi^{n,1}),
\]
with $\pi^{0,1}=\tilde \pi$. This procedure defines a continuous mapping from $\Sss^{N+1}$ into itself that associates
to the sequence $\pi^{n,0}$ the new sequence of probability vectors $\pi^{n,1}$.
Therefore, by Brower's fixed point theorem, it has a fixed point, which corresponds to a solution to the problem.
\end{proof}

\subsection{Uniqueness}
\label{usitvp}

As for stationary solution we adapt Lasry and Lions monotonicity arguments to obtain uniqueness of solutions.

\begin{teo}
\label{teoE}
Suppose assumptions \ref{hp4}, \ref{hp7} and \ref{hp8} hold.
Let  $\{(\pi^n,V^n) \;  ; \;0\leq n \leq N \} \, $ and $\{(\tilde\pi^n,\tilde V^n) \;  ; \;0\leq n \leq N \} \, $ be  solutions of the the mean field game with $\pi^0=\tilde \pi^0$ and $V^N=\tilde V^N$.
Then $\pi^n=\tilde\pi^n$, and $V^n=\tilde V^n$, for all $0\leq n \leq N$.
\end{teo}
\begin{proof} We have
$$\Gg_{\pi^n}(V^{n+1})=V^{n}, \;\;\; \;\;\;\Kk_{V^{n+1}}(\pi^n)=\pi^{n+1},$$ and
$$\Gg_{\tilde \pi^n}(\tilde V^{n+1})=\tilde V^{n}, \;\;\; \;\;\;\Kk_{\tilde V^{n+1}}(\tilde \pi^n)=\tilde \pi^{n+1}.$$
Then
\begin{align*}0=&\sum_{n=0}^{N-1}(V^{n+1}-\tilde V^{n+1})\cdot [(\Kk_{V^{n+1}}(\pi^n)-\pi^{n+1})-(\Kk_{\tilde V^{n+1}}(\tilde \pi^n)-\tilde \pi^{n+1})]\\
&
+\sum_{n=0}^{N-1}(\pi^n-\tilde \pi^n)\cdot [(\Gg_{\tilde \pi^n}(\tilde V^{n+1})-\tilde V^{n})-(\Gg_{\pi^n}(V^{n+1})-V^{n})].
\end{align*}
Note that $(V^N-\tilde V^N)\cdot (\tilde \pi^N-\pi^N)=0$ and $( \pi^0-\tilde\pi^0)\cdot (V^0-\tilde V^0)=0$.
Thus rewriting the identity above we have
$$0=\sum_{n=0}^{N-1} \pi^n\cdot (\Gg_{ \pi^n}(\tilde V^{n+1})-\Gg_{ \pi^n}( V^{n+1}))+\Kk_{V^{n+1}}(\pi^n)\cdot (V^{n+1}-\tilde V^{n+1}) \;+ $$
$$ +\sum_{n=0}^{N-1} \tilde\pi^n\cdot (\Gg_{\tilde \pi^n}( V^{n+1})-\Gg_{\tilde \pi^n}( \tilde V^{n+1}))+\Kk_{\tilde V^{n+1}}(\tilde \pi^n)\cdot
(\tilde V^{n+1}- V^{n+1}) \;+$$
$$+ \sum_{n=0}^{N-1} \pi^n\cdot (\Gg_{ \tilde\pi^n}(\tilde V^{n+1})-\Gg_{ \pi^n}( \tilde V^{n+1}))+\tilde \pi^n\cdot (\Gg_{ \pi^n}( V^{n+1})-\Gg_{ \tilde \pi^n}(  V^{n+1})).$$
Now, using  \eqref{kav3}, we have, for each $0\leq n\leq N-1$,
$$\pi^n\cdot (\Gg_{ \pi^n}(\tilde V^{n+1})-\Gg_{ \pi^n}( V^{n+1}))+\Kk_{V^{n+1}}(\pi^n)\cdot (V^{n+1}-\tilde V^{n+1})\leq 0$$ and similarly for the terms of the second line. In the third line we have
$$\pi^n\cdot (\Gg_{ \tilde\pi^n}(\tilde V^{n+1})-\Gg_{ \pi^n}( \tilde V^{n+1}))+\tilde \pi^n\cdot (\Gg_{ \pi^n}( V^{n+1})-\Gg_{ \tilde \pi^n}(  V^{n+1}))\leq -\gamma \|\pi^n-\tilde \pi^n\|^2.$$
Hence $$\sum_{n=0}^{N-1} \gamma \|\pi^n-\tilde \pi^n\|^2\leq 0.$$
This implies $\pi^n=\tilde\pi^n$ for all $0\leq n \leq N$.

To obtain  $V^n=\tilde V^n$ for all $0\leq n \leq N$, we just have to use
$V^N=\tilde V^N$ and apply iteratively the operator
 $\Gg_{\pi^n}(V^{n+1})=V^{n}$ and  $\Gg_{\tilde \pi^n}(\tilde V^{n+1})=\tilde V^{n}$, from $n=N-1$ to $n=0$.
\end{proof}

%
%



\section{Convergence to equilibrium}
\label{cvte}

In this last section we discuss the main contribution of this paper, namely the exponential convergence to equilibrium
for the initial-terminal value problem.
Our setting is the following: consider a
initial-terminal value problem with initial data
$\pi^{-N}$ and terminal data $V^N$. We will now study conditions under which
$\pi^0\to \bar \pi$ and $V^0\to \bar V$ where $(\bar \pi, \bar V)$ are stationary solutions, as $N\to \infty$.
In fact we will show this is true if assumptions \ref{hp4}, \ref{hp7}, \ref{hp8}, \ref{hp10} and \ref{hp11} hold.

\subsection{A-priori bounds}
\label{apb2}

We start by establishing some useful a-priori bounds.
\begin{lem}\label{bound}
Suppose assumption \ref{hp4}, \ref{hp7}, and \ref{hp11} holds.
Let  $\{(\pi^n,V^n) \;  ; \;-N\leq n \leq N \} \, $ and $\{(\tilde\pi^n,\tilde V^n) \;  ; \;-N\leq n \leq N \} \, $  be two  solutions of  the mean field game. Then we have
\begin{equation}\label{lema_conveq} \|\tilde V^{-N}-V^{-N} \| \leq  \|\tilde V^N-V^N \| +N2K .
\end{equation}
\end{lem}
\begin{proof}
Applying  proposition \ref{concav} we have that $$\Gg_{\pi^{N-1}}(\tilde V^N)_i -\Gg_{\pi^{N-1}}(V^N)_i \leq \sum_j P_{ij}^{ V^N, \pi^{N-1}}(\tilde V^N_j-V^N_j).$$
Then
$$\Gg_{\pi^{N-1}}(\tilde V^N)_i - \Gg_{\tilde \pi^{N-1}}(\tilde V^N)_i+\tilde V^{N-1}_i -V^{N-1}_i\leq \|\tilde V^N-V^N \|, $$
also we have
\begin{align*}
&\Gg_{\pi^{N-1}}(\tilde V^N)_i- \Gg_{\tilde \pi^{N-1}}(\tilde V^N)_i \\&\qquad \leq \sum_j
[c_{ij}(\pi^{N-1},P_{i\cdot}^{\tilde V^N,\tilde \pi^{N-1}})-c_{ij}(\tilde \pi^{N-1},P_{i\cdot}^{\tilde V^N,\tilde \pi^{N-1}})]P_{ij}^{\tilde V^N,\tilde \pi^{N-1}}\leq K.
\end{align*}
Hence,
$$\tilde V^{N-1}_i -V^{N-1}_i \leq \|\tilde V^N-V^N \| +K$$
Exchanging the roles of $(\pi^{N-1},V^N)$ and $(\tilde\pi^{N-1},\tilde V^N)$  we get
$$ V^{N-1}_i -\tilde V^{N-1}_i\leq \|\tilde V^N-V^N \| +K,$$ thus
$$\|\tilde V^{N-1}-V^{N-1} \| \leq  \|\tilde V^N-V^N\| +K $$
Reasoning by induction we obtain that
$$ \|\tilde V^{-N}-V^{-N} \| \leq  \|\tilde V^N-V^N \| +N2K .$$
 \end{proof}

\subsection{Exponential convergence}

We recover the proof of theorem \ref{teoE} to obtain an important estimate:
\begin{pro}
\label{trendtoeq}
Suppose assumptions \ref{hp4}, \ref{hp7}, \ref{hp8} and \ref{hp10} hold.
Let  $\{(\pi^n,V^n) \;  ; \;-N\leq n \leq N \} \, $ and $\{(\tilde\pi^n,\tilde V^n) \;  ; \;-N\leq n \leq N \} \, $ be  solutions to the
 mean field game. Let $C=1/\gamma$.
 Then
\begin{align*}
&\sum_{n=-N+1}^{N-1} \|\pi^n-\tilde \pi^n\|^2+\|V^{n}-\tilde V^{n}\|_{\#}^2\\ &\quad \leq C \left(
\|\pi^N-\tilde \pi^N\|^2+\|V^N-\tilde V^N\|_{\#}^2+\|\pi^{-N}-\tilde \pi^{-N}\|^2+\|V^{-N}-\tilde V^{-N}\|_{\#}^2\right)
\end{align*}
\end{pro}
\begin{proof}
As before, observe that
\begin{align*}
&0=\sum_{n=-N}^{N-1}(V^{n+1}-\tilde V^{n+1})\cdot[(\Kk_{V^{n+1}}(\pi^n)-\pi^{n+1})-(\Kk_{\tilde V^{n+1}}(\tilde \pi^n)-\tilde \pi^{n+1})]\\
&+\sum_{n=-N}^{N-1}(\pi^n-\tilde \pi^n)\cdot[(\Gg_{\tilde \pi^n}(\tilde V^{n+1})-\tilde V^{n})-(\Gg_{\pi^n}(V^{n+1})-V^{n})].
\end{align*}
Then rewriting the equation above we have
\begin{align*}
&\sum_{n=-N}^{N-1} (V^{n+1}-\tilde V^{n+1}) \cdot(\pi^{n+1}-\tilde \pi^{n+1})-(V^{n}-\tilde V^{n})\cdot (\pi^{n}-\tilde \pi^{n})\\
&=\sum_{n=-N}^{N-1} \pi^n\cdot (\Gg_{ \pi^n}(\tilde V^{n+1})-\Gg_{ \pi^n}( V^{n+1}))+\Kk_{V^{n+1}}(\pi^n)\cdot(V^{n+1}-\tilde V^{n+1}) \;+\\
& +\sum_{n=_N}^{N-1} \tilde\pi^n\cdot (\Gg_{\tilde \pi^n}( V^{n+1})-\Gg_{\tilde \pi^n}( \tilde V^{n+1}))+\Kk_{\tilde V^{n+1}}(\tilde \pi^n)\cdot(\tilde V^{n+1}- V^{n+1}) \;+\\
&+ \sum_{n=-N}^{N-1} \pi^n\cdot (\Gg_{ \tilde\pi^n}(\tilde V^{n+1})-\Gg_{ \pi^n}( \tilde V^{n+1}))+\tilde \pi^n\cdot (\Gg_{ \pi^n}( V^{n+1})-\Gg_{ \tilde \pi^n}(  V^{n+1})).
\end{align*}
Now, for each $-N\leq n\leq N-1$, we have that
$$\pi^n\cdot (\Gg_{ \pi^n}(\tilde V^{n+1})-\Gg_{ \pi^n}( V^{n+1}))+\Kk_{V^{n+1}}(\pi^n)\cdot(V^{n+1}-\tilde V^{n+1})\leq
-\gamma  \|V^{n+1}-\tilde V^{n+1}\|_{\#}^2,
$$ and similarly for the terms of the second line. In the third line we have
$$\pi^n\cdot (\Gg_{ \tilde\pi^n}(\tilde V^{n+1})-\Gg_{ \pi^n}( \tilde V^{n+1}))+\tilde \pi^n\cdot (\Gg_{ \pi^n}( V^{n+1})-\Gg_{ \tilde \pi^n}(  V^{n+1}))\leq -\gamma \|\pi^n-\tilde \pi^n\|^2.$$
Consequently
\begin{align*}
&\sum_{n=-N}^{N-1} \|\pi^n-\tilde \pi^n\|^2
+2\|V^{n+1}-\tilde V^{n+1}\|_{\#}^2
\\
&\quad \leq
\frac{1}{\gamma} \left(
(\pi^N-\tilde \pi^N)\cdot(\tilde V^N- V^N)+(\pi^{-N}-\tilde \pi^{-N})\cdot(V^{-N}-\tilde V^{-N})\right).
\end{align*}

Note that, if $c$ is the constant vector  then $(\pi^k-\tilde\pi^k)\cdot \mu=0$, where $k=N, -N$. Also, there exists $\mu_k$ such that $\|\tilde V^k- V^k \|=| \tilde V^k- V^k+\mu_k|.$
Hence
\begin{align*}
&\sum_{n=-N+1}^{N-1} \|\pi^n-\tilde \pi^n\|^2
+\|V^{n}-\tilde V^{n}\|_{\#}^2
\\
&\quad \leq
C \left(
\|\pi^N-\tilde \pi^N\|^2+\|V^N-\tilde V^N\|_{\#}^2+\|\pi^{-N}-\tilde \pi^{-N}\|^2+\|V^{-N}-\tilde V^{-N}\|_{\#}^2\right),
\end{align*}
if we denote $C=1/\gamma$.
\end{proof}

Define $f_0= \|\pi^0-\tilde \pi^0\|^2
+ \|V^{0}-\tilde V^{0}\|_{\#}^2$, and, for $n>0$
\[
f_n=\|\pi^n-\tilde \pi^n\|^2
+\|V^{n}-\tilde V^{n}\|_{\#}^2+\|\pi^{-n}-\tilde \pi^{-n}\|^2
+\|V^{-n}-\tilde V^{-n}\|_{\#}^2.
\]
The previous proposition implies
\begin{equation}
\label{estimate}
\sum_{n=0}^{N-1} f_n \leq C f_N.
\end{equation}

Note that the previous proposition and lemma \ref{bound} imply
\begin{equation} \label{fN}
f_N\leq \|\pi^N-\tilde \pi^N\|^2+\|\pi^{-N}-\tilde \pi^{-N}\|^2+\|V^{N}-\tilde V^{N}\|_{\#}^2+(\|V^{N}-\tilde V^{N}\|+N2K)^2\,.
\end{equation}
The next lemma is the only missing tool to get exponential decay:
\begin{lem}\label{numeros}
Suppose $f_n\geq 0$ and that
\begin{equation}
\label{22B}
\sum_{n=0}^{N-1} f_n \leq C f_N.
\end{equation}
Then
\[
f_0\leq C\left(\frac{C}{C+1}\right)^{N-1} f_N.
\]
\end{lem}
\begin{proof}
The proof follows by induction. The case $N=1$ is simply a particular case of \eqref{22B}:
\[
f_0\leq C f_1.
\]
Now, observe that
\begin{align*}
C f_{N+1}&\geq \sum_{n=0}^{N} f_n\geq  f_0+\sum_{n=1}^{N} \frac 1 C \left(\frac{C+1}{C}\right)^{n-1} f_0\\
&=f_0\left[
1+ \frac 1 C \frac{\left(\frac{C+1}{C}\right)^N-1}{\frac{C+1}{C}-1}
\right]
=f_0\left[
1+ \left(\frac{C+1}{C}\right)^N-1
\right]=f_0 \left(\frac{C+1}{C}\right)^N,
\end{align*}
which ends the proof.
\end{proof}





\begin{teo}
\label{teoF}
Suppose assumptions \ref{hp4}, \ref{hp7}, \ref{hp8}, \ref{hp10} and \ref{hp11} hold.
Fix  $\tilde V, \tilde \pi $. Given $N>0$, denote by $(\pi^0_N,V^0_N)$  the solution of the mean field game at time 0 that has initial distribution $\pi^{-N}=\tilde \pi$ and terminal cost $V^{N}=\tilde V$.


Then, as $N\to \infty$
\[
V_N^0\to \bar V \;( \mbox{ in } \re^d/\Rr), \quad \pi_N^0\to \bar \pi
\]
where $\bar V$ and $\bar \pi$ is the unique stationary solution.
\end{teo}
\begin{proof}
By theorem \ref{teoE}, we can define, for each $N$, a map $\Xi_N:\mathbb{S}\times \re^d/\Rr \to \mathbb{S}\times \re^d/\Rr$, that
associates to each pair $(\tilde \pi,\tilde V)$ the pair $(\pi^0_N,V^0_N)$.
Here $\mathbb{S}\times \re^d/\mathbb{R}$ is given the product topology where in $\re^d/\mathbb{R}$ we consider the norm $\|U\|_{\#}=\inf_{k\in\re}\|U+k  \|$.


Now, lemma \ref{numeros} and equation (\ref{fN}) show that, for any two pairs $(V,\pi)$ and $(\tilde V, \tilde \pi)$, we have
$$\Xi_N(V,\pi)-\Xi_N(\tilde V,\tilde \pi) \to (0,0)$$ as $N\to \infty$.
Thus,
$$\Xi_N(V,\pi)-(\bar V,\bar \pi) = \Xi_N(V,\pi)-\Xi_N(\bar V,\bar \pi) \to (0,0)$$ as $N\to \infty$.
\end{proof}






\end{document}